\documentclass[11pt]{amsart}

\usepackage{amsmath,amssymb,latexsym,soul,cite,mathrsfs}

\usepackage{color,enumitem,graphicx}
\usepackage[colorlinks=true,urlcolor=blue,
citecolor=red,linkcolor=blue,linktocpage,pdfpagelabels,
bookmarksnumbered,bookmarksopen]{hyperref}
\usepackage[english]{babel}

\usepackage[left=2.9cm,right=2.9cm,top=2.8cm,bottom=2.8cm]{geometry}
\usepackage[hyperpageref]{backref}

\usepackage[colorinlistoftodos]{todonotes}
\makeatletter
\providecommand\@dotsep{5}
\def\listtodoname{List of Todos}
\def\listoftodos{\@starttoc{tdo}\listtodoname}
\makeatother

\numberwithin{equation}{section}

\newtheorem{theorem}{Theorem}[section]

\newtheorem{lemma}[theorem]{Lemma}

\newtheorem{claim}[theorem]{Claim}

\newtheorem{remark}{Remark}

\begin{document}

\title[A global minimization trick ...]
{A global minimization trick to solve some classes of Berestycki-Lions type problems}
\author{Claudianor O. Alves }
\address[Claudianor O. Alves ]
{\newline\indent Unidade Acad\^emica de Matem\'atica
\newline\indent 
Universidade Federal de Campina Grande 
\newline\indent
58429-970, Campina Grande - PB, Brazil}

\pretolerance10000


\begin{abstract}
\noindent In this paper we show an abstract theorem that can be used to prove the existence of solution for a class of elliptic equation considered in Berestycki-Lions \cite{berest} and related problems. Moreover, we use the abstract theorem to show that a class of zero mass problems has multiple solutions, which is new for this type of problem. 
\end{abstract}

\thanks{ C. O. Alves was partially
supported by  CNPq/Brazil 304804/2017-7 }

\maketitle

\section{Introduction}

At the last years a lot of authors have dedicated a special attention for the existence of positive solution for elliptic problems of the type  
\begin{equation}\label{BL}
- \Delta u   = g(u), \quad \mbox{in} \quad \mathbb{R}^N, 
\end{equation}
where $N \geq 2$, $\Delta$ denotes the Laplacian operator and $g$ is a continuous function verifying some conditions. The main motivation to study the above problem comes from of the seminal paper due to Berestycki and Lions \cite{berest}, which has considered the existence of solution for  (\ref{BL}) by assuming that $N \geq 3$ and the following conditions on $g$:
$$
- \infty < \liminf_{s \to 0^+}\frac{g(s)}{s} \leq \limsup_{s \to 0^+}\frac{g(s)}{s}\leq -m<0, \leqno{(g_1)} 
$$
$$
\limsup_{s \to +\infty}\frac{g(s)}{s^{2^{*}-1}}\leq 0, \leqno{(g_2)}
$$
$$
\mbox{there is} \quad \xi>0 \, \, \mbox{such that} \,\, G(\xi)>0, \leqno{(g_3)}
$$
where $G(s)=\int_{0}^{s}g(t)\,dt$.

In \cite{BGK}, Berestycki, Gallouet and Kavian have studied the case where $N=2$ and the nonlinearity $g$ possesses an exponential growth of the type
$$
\limsup_{s \to 0^+}\frac{g(s)}{e^{\beta s^2}}=0, \quad \forall \beta >0.
$$

In the two  above mentioned papers, the authors have used the variational method to prove the existence of solution for (\ref{BL}). The main idea is to solve the minimization problems 
$$
\min \left\{\frac{1}{2}\int_{\mathbb{R}^N}|\nabla u|^{2} \,dx \,:\, \int_{\mathbb{R}^N}G(u)\,dx=1  \right\}  
$$
and 
$$
\min \left\{\frac{1}{2}\int_{\mathbb{R}^N}|\nabla u|^{2}\,dx \,:\, \int_{\mathbb{R}^N}G(u)\,dx=0  \right\}  
$$
for $N \geq 3$ and $N=2$ respectively. After that, the authors showed that the minimizer functions of the above problems are in fact ground state solutions of (\ref{BL}). By a {\it ground state solution}, we mean a solution $u \in H^{1}(\mathbb{R}^N) \setminus \{0\}$ that satisfies
$$
E(u) \leq E(v) \quad \mbox{for all nontrival solution} \  v  \ \text{of} \ (\ref{BL}),
$$ 
where $E:H^{1}(\mathbb{R}^N) \to \mathbb{R}$ is the energy functional associated to (\ref{BL}) given by
$$
E(u)=\frac{1}{2}\int_{\mathbb{R}^N}|\nabla u|^{2}\,dx - \int_{\mathbb{R}^N}G(u)\,dx.
$$

After, Jeanjean and Tanaka in \cite{JJTan} showed that the mountain pass level of  $E$ is a critical level and it is 
indeed the lowest critical level.

A version of the problem (\ref{BL}) for the critical case have been  made in Alves, Souto and Montenegro \cite{AlvesSoutoMontenegro} for $N \geq 3$ and $N=2$, see also Zhang and Zhou \cite{ZZ} for $N=3$. The reader can found in Alves, Figueiredo and Siciliano \cite{AGS}, Chang and Wang \cite{ChangWang} and Zhang, do \'O and Squassina \cite{ZOS} the same type of results involving the fractional Laplacian operator, more precisely, for a problem like 
\begin{equation}\label{fracionario}
(- \Delta)^{\alpha}u = g(u), \quad \mbox{in} \quad \mathbb{R}^N, 
\end{equation}
with $\alpha \in (0,1)$ and $N \geq 1$. For more details about this subject, we would like to cite the references found in the above mentioned papers.

Recently, Alves, Duarte and Souto \cite{ADS} have proved an abstract theorem that was used to solve a large class of Berestycki-Lions type problems, which includes Anisotropic operator, Discontinuous nonlinearity, etc.. In that paper, the authors have used the deformation lemma together with a notation of Pohozaev set to prove the abstract theorem.

In the present paper, we show a new abstract theorems, see Theorems \ref{AT1}, \ref{AT2} and \ref{AT3}, which can be used to prove the existence of solution for (\ref{BL}) and related problems. These theorems are totally different from that proved in \cite{ADS} and their proofs are very simple, however Theorem \ref{AT2} permits to show the existence of multiple solutions for a class of zero mass  problem that is a new result for this class of problem. The present article completes the study made in \cite{ADS}, in the sense that we can consider other class of problems, for example, our approach can be applied to study a version of Berestycki-Lions type problems for elliptic system,  which were not considered until moment in the literature, see Section 6.

\section{A global minimization trick: An Abstract theorem}

In order to prove our abstract theorem, it is necessary to fix some notations. In what follows, we say that two functionals $\Phi,\Psi:X \to \mathbb{R}$ are {\bf Admissible Functionals} with relation to a Banach space $X$ with norm $\|\,\,\,\,\|$, denoted by $(\Phi,\Psi) \in (\mathcal{AF})_X$, if the following properties hold: 
$$
\Phi,\Psi \quad \mbox{ take bounded set of} \quad X \quad \mbox{into bounded set of} \quad \mathbb{R}.  \leqno(AF_1)
$$
$$
\Phi(u) > 0 \quad \forall u \in X \setminus \{0\} \quad \mbox{and} \quad \Phi(0)=0. \leqno{(AF_2)}
$$ 
\noindent ${(AF_3)}$ The functionals $\Phi$ and $\Psi$ have the {\it \bf Absorption Property} between them, that is, $\Phi \in C^{1}(X,\mathbb{R})$, $\Psi$ is a Locally Lipschitzian functional, and the following property occurs: If $\lambda >0$ and $u \in X \setminus \{0\}$ satisfy  
$$
0 \in \Phi'(u)-\lambda \partial \Psi(u),
$$
then there is $w=w(\lambda,u) \in X \setminus \{0\}$ such that  
$$
0 \in \Phi'(w)-\partial \Psi(w). 
$$
Hereafter, if $I:X \to \mathbb{R}$ is a Locally Lipschitzian functional, we denote by $\partial I(u)$ generalized gradients of $I$ at $u$. Moreover, we recall that $u \in X$ is a 
critical point for $I$ if $ 0 \in \partial I(u)$. 	The reader can find more details about this subject in Chang \cite{Chang} and Clarke \cite{clarke}.

In the sequel, a Locally Lipschitzian functional $J:X \to \mathbb{R}$ is {\it \bf Almost Lower Semicontinuous}  with relation to $X$, denoted by $J \in (\mathcal{ALS})_X$, if for any bounded sequence $(u_n) \subset X$, there is a bounded sequence $(v_n) \subset X$ and $v \in X$ with $v_n \rightharpoonup v$ in $X$ such that
$$
J(u_n) \geq J(v_n), \quad \forall n \in \mathbb{N}
$$ 
and 
$$
\liminf_{n \to +\infty}J(v_n) \geq J(v).
$$
Furthermore, we will say that a $C^{1}$-function $h:[0,+\infty) \to [0,+\infty)$ is {\it \bf compatible} with the functionals $\Phi,\Psi:X \to \mathbb{R}$, denoted by $h \in \mathcal{C}_{\Phi,\Psi}$, if the three properties below are satisfied:  
$$
h'(t)>0 \quad \mbox{for} \quad t>0 \quad \mbox{and} \quad h(0)=0. \leqno{(H_1)}
$$
\noindent $(H_2)$ \, The functional $J_h:X \to \mathbb{R}$ given by 
$$
J_h(u)=h(\Phi(u))-\Psi(u), \quad \forall u \in X 
$$
is coercive and belongs to $(\mathcal{ALS})_X$.

\noindent $(H_3)$ \, There is $e \in X \setminus\{0\}$ such that  
$$
\frac{\Psi(e)}{h(\Phi(e))}>1. 
$$

\begin{theorem}(A global minimization trick) \label{AT1} Let $X$ be a reflexive Banach space and consider three functionals $J,\Phi,\Psi:X \to \mathbb{R}$ such that 
	$$
	J(u)=\Phi(u)-\Psi(u), \quad \forall u \in X,
	$$	
	with $(\Phi,\Psi) \in (\mathcal{AF})_X$. If $\mathcal{C}_{\Phi,\Psi}$ is not empty, then functional $J$ has a nontrivial critical point.
\end{theorem} 
\begin{proof}\, Since by hypothesis $\mathcal{C}_{\Phi,\Psi} \not= \emptyset$, let us take $h \in \mathcal{C}_{\Phi,\Psi}$ and the functional $J_h:X \to \mathbb{R}$ given by
$$
J_h(u)=h(\Phi(u))-\Phi(u), \quad \forall u \in X.
$$
First of all, as $J_h$ is coercive, there is $R>0$ such that 
$$
J_h(u) \geq 1, \quad \mbox{for} \quad \|u\| > R.
$$
From $(AF_1)$, 
$$
J_h(u) \geq -M_1, \quad \mbox{for} \quad \|u\|\leq R, 
$$
for some $M_1>0$. The above analysis yields $J_h$ is bounded from below in $X$. Thus, there is a sequence $(u_n) \subset X$ satisfying
$$
\lim_{n \to +\infty}J_h(u_n)=\inf_{u \in X}J_h(u)=J^{\infty}_h.
$$
Since $J_h$ is coercive, it follows that $(u_n)$ is a bounded sequence in $X$. Recalling that $J_h \in (\mathcal{ALS})_X$, there are $(v_n) \subset X$ and $v \in X$ such that
$$
\lim_{n \to +\infty}J_h(u_n)\geq \lim_{n \to +\infty}J_h(v_n)=J_h(v), 
$$
showing that
$$
J_{h}(v)=J^{\infty}_h.
$$
Therefore, owing  \cite[Proposition 6]{clarke2}, we can infer that $0 \in \partial J_h(v)$, and also, 
$$
0 \in h'(\Phi(v))\Phi'(v)-\partial \Psi(v).
$$
On the other hand,  by $(H_2)$,  
$$
J_h(e)=J_h(u)=h(\Phi(e))-\Phi(e)<0,
$$
from where it follows that
$$
J_{h}(v)=J^{\infty}_h<0,
$$
and so, $v \not=0$. Setting $\lambda =\frac{1}{h'(\Phi(v))}>0$, we get
$$
0 \in \Phi'(v)-\lambda \partial \Psi(v).
$$
By employing $(AF_3)$, there is $u \in X \setminus \{0\}$ such that 
$$
0 \in \partial J(u)=\Phi'(u)- \partial \Psi(u).
$$
Finally, recalling that $\partial J(u)=\Phi'(u)- \partial \Psi(u)$, we deduce that $0 \in \partial J(u)$, which proves the theorem. 
\end{proof}

\begin{remark} \label{RE} Before concluding this section, we would like point out that if $\Psi \in C^{1}(X,\mathbb{R})$, then the conclusion of Theorem \ref{AT1} can be rewritten of the form  
$$
\Phi'(u)-\Psi'(u)=0,
$$
implying that $u \in X$ is critical point for $J$ in the usual sense.  
\end{remark}

\section{Application: The positive mass case}

In this section, we will apply Theorem \ref{AT1} to prove the existence of a nontrivial solution to  (\ref{BL}).

Arguing as \cite{berest}, we can extend function $g$ to whole $\mathbb{R}$ with $g=g_1-g_2$, where $g_1,g_2$ are odd functions that satisfy  
$$
\lim_{s \to 0}\frac{g_1(s)}{s}=\lim_{s \to 0}\frac{g_1(s)}{|s|^{2^*-1}}=0
$$
and
$$
g_2(s) \geq m s, \quad \forall s \geq 0.
$$

The energy functional associated with  (\ref{BL}), denoted by $J: H^{1}(\mathbb{R}^N) \to \mathbb{R}$, has the form
$$
J(u)=\frac{1}{2}\int_{\mathbb{R}^N}|\nabla u|^{2}\,dx-\int_{\mathbb{R}^N}G(u)\,dx.
$$
Hereafter, we designate by $\Phi,\Psi:H^{1}(\mathbb{R}^N) \to \mathbb{R}$ the functionals given by 
\begin{equation} \label{phi1}
\Phi(u)=\frac{1}{2}\int_{\mathbb{R}^N}|\nabla u|^{2}\,dx
\end{equation}
and
\begin{equation} \label{psi1}
\Psi(u)=\int_{\mathbb{R}^N}G(u)\,dx.
\end{equation}
It is clear that $\Phi$ and $\Psi$ verify $(AF_1)$ and $(AF_2)$. Next, we will show that $(AF_3)$ is true here.

\begin{lemma} \label{L1} Let $\lambda >0$ and $u \in H^{1}(\mathbb{R}^N)$ be a nontrivial solution of 
$$
\int_{\mathbb{R}^N}\nabla u \nabla w\,dx-\lambda \int_{\mathbb{R}^N}g(u)w\,dx=0, \quad \forall w \in H^{1}(\mathbb{R}^N).
$$
Then $v(x)=u(x/\sqrt{\lambda})$ is a nontrivial solution of    
$$
\int_{\mathbb{R}^N}\nabla v \nabla w\,dx-\int_{\mathbb{R}^N}g(v)w\,dx=0, \quad \forall w \in H^{1}(\mathbb{R}^N).
$$
Hence, by Remark \ref{RE},  $(AF_3)$ also occurs.
\end{lemma}
\begin{proof} By regularity theory we know that $u \in W_{loc}^{2,p}(\mathbb{R}^N) \cap H^{1}(\mathbb{R}^N)$ for some $p>1$. Then,   
$$
-\Delta u(x)=\lambda g(u(x)), \quad a.e. \;\;\; x \in \mathbb{R}^N.
$$ 
Setting $v(x)=u(x/\sqrt{\lambda})$, a direct computation gives that $v$ is a nontrivial solution of   
$$
-\Delta v=g(v), \quad a.e. \;\;\; x \in \mathbb{R}^N,
$$ 
and the proof is complete. 
\end{proof}

Now, we will see that $\mathcal{C}_{\Phi,\Psi}$ is not empty. Have this mind, let us fix $k > \frac{N}{N-2}$ and the function $h:[0,+\infty) \to \mathbb{R}$ given by 
\begin{equation} \label{h1}
h(t)=t^k, \quad \forall t \geq 0.
\end{equation}
We claim that $h \in \mathcal{C}_{\Phi,\Psi}$. Indeed, let $J_h:H^{1}(\mathbb{R}^N) \to \mathbb{R}$ be the functional defined by
$$
J_h(u)=h(\Phi(u))-\Psi(u),
$$
that is,
\begin{equation} \label{jh}
J_h(u)=\frac{1}{k}\left(\frac{1}{2}\int_{\mathbb{R}^N}|\nabla u|^{2}\,dx \right)^{k}-\int_{\mathbb{R}^N}G(u)\,dx.
\end{equation}
From definition of $g$, it is easily checked that $J_h \in C^{1}(H^{1}(\mathbb{R}^N),\mathbb{R})$ with
$$
J'_h(u)=\left(\frac{1}{2}\int_{\mathbb{R}^N}|\nabla u|^{2}\,dx \right)^{k-1}\int_{\mathbb{R}^N}\nabla u \nabla v\,dx-\int_{\mathbb{R}^N}g(u)v\,dx, \quad \forall u,v \in H^{1}(\mathbb{R}^N).
$$

\begin{lemma} \label{L2} The functional $J_h$ belong to $(\mathcal{ALS})_{H^{1}(\mathbb{R}^N)}$.  
\end{lemma}
\begin{proof} Let $(u_n) \subset H^{1}(\mathbb{R}^N)$ be a bounded sequence.  Since $J_h(u)=J_h(|u|)$ for all $u \in H^{1}(\mathbb{R}^N)$, we have that  
$$
J_h(u_n)=J_h(|u_n|), \quad \forall n \in \mathbb{N}.
$$
If $u_n^*$ denotes the Schwarz symmetrization of $|u_n|$, there hold 
$$
\int_{\mathbb{R}^N}|\nabla |u_n||^{2}\,dx \geq \int_{\mathbb{R}^N}|\nabla u_n^*|^{2}\,dx \quad \mbox{and} \quad \int_{\mathbb{R}^N}F(|u_n|)\,dx=\int_{\mathbb{R}^N}F(u_n^*)\,dx,
$$
which lead to 
$$
J_h(u_n) \geq J_h(u_n^*), \quad \forall n \in \mathbb{N}.
$$
By boundedness of $(u_n^*)$, we can assume that $u_n^* \rightharpoonup u_0$ in $H^{1}(\mathbb{R}^N)$ for some $u_0 \in H_{rad}^{1}(\mathbb{R}^N)$. Therefore, the conditions on $g$ combined with Strauss' Lemma ensure that 
$$
\lim_{n \to +\infty}\int_{\mathbb{R}^N}G_1(u_n^*)\,dx = \int_{\mathbb{R}^N}G_1(u_0)\,dx
$$
and
$$
\liminf_{n \to +\infty}\int_{\mathbb{R}^N}G_2(u_n^*)\,dx \geq \int_{\mathbb{R}^N}G_2(u_0)\,dx.
$$
Hence, 
$$
\liminf_{n \to +\infty}\int_{\mathbb{R}^N}(-G(u_n^*))\,dx=\liminf_{n \to +\infty}\int_{\mathbb{R}^N}G_2(u_n^*)\,dx-\lim_{n \to +\infty}\int_{\mathbb{R}^N}G_1(u_n^*)\,dx
$$
and so,
$$
\liminf_{n \to +\infty}\int_{\mathbb{R}^N}(-G(u_n^*))\,dx \geq \int_{\mathbb{R}^N}G_2(u_0)\,dx-\int_{\mathbb{R}^N}G_1(u_0)\,dx=\int_{\mathbb{R}^N}(-G(u_0))\,dx.
$$
Recalling that
$$
\liminf_{n \to +\infty}\int_{\mathbb{R}^N}|\nabla u_n^*|^{2}\,dx \geq \int_{\mathbb{R}^N}|\nabla u_0|^{2}\,dx,
$$
we obtain
$$
\liminf_{n \to +\infty}J_h(u_n^*)\geq \liminf_{n \to +\infty}\int_{\mathbb{R}^N}|\nabla u_n^*|^{2}\,dx + \liminf_{n \to +\infty}\int_{\mathbb{R}^N}(-G(u_n))\,dx \geq J_h(u_0),
$$
which completes the proof. 
\end{proof}

\begin{lemma} \label{L3} The conditions $(H_2)-(H_3)$ are satisfied for $\Phi, \Psi$ and $h$ given in (\ref{phi1}),(\ref{psi1}) and (\ref{h1}) respectively. 
\end{lemma}
\begin{proof}\, By assumptions on $g$, 
$$
G(s) \leq -m|s|^{2}+C|s|^{2^*}, \quad \forall s \in \mathbb{R}. 
$$
Consequently, by Sobolev embedding, 
$$
\Psi(u) \leq -m|u|^{2}_{L^{2}(\mathbb{R}^N)}+C|\nabla u|^{2^*}_{L^{2}(\mathbb{R}^N)}, \quad \forall u \in H^{1}(\mathbb{R}^N).
$$
On the other hand, by definition of $h$, 
$$
h(\Phi(u))=\frac{1}{k}\left(\frac{1}{2}|\nabla u|^{2}_{L^{2}(\mathbb{R}^N)}\right)^{k}.
$$
Since $k > \frac{N}{N-2}$, a simple computation gives
$$
\lim_{\|u\| \to +\infty}J_h(u)=h(\Phi(u))-\Psi(u)=+\infty, 
$$
which establishes the coercivity of $J_h$. This together with Lemma \ref{L2} shows that $(H_2)$ is valid. Now, we are going to prove $(H_3)$. Have this in mind, firstly we recall that $(g_3)$ guarantees the existence of a function $\phi \in C_0^{\infty}(\mathbb{R}^N)$ with
$$
\int_{\mathbb{R}^N}G(\phi)\,dx>0.
$$  
For each $t>0$, the function $\phi_t(x)=\phi(x/t)$ belongs to $H^{1}(\mathbb{R}^N)$ with
$$
\Psi(\phi_t)=t^{N}\Psi(\phi) \quad \mbox{and} \quad \Phi(\phi_t)=t^{N-2}\Phi(\phi).
$$
This together with the fact that $k>\frac{N}{N-2}$ yields
$$
\lim_{t\to 0^+ }\frac{\Psi(\phi_t)}{h(\Phi(\phi_t))}=\lim_{t\to 0^+ }\frac{t^N\Psi(\phi)}{t^{k(N-2)}h(\Phi(\phi))}=+\infty.
$$
Thereby, $(H_3)$ holds with $e=\phi_t$ and $t$ small enough. 
\end{proof}

As a consequence of Lemmas \ref{L1}-\ref{L3}, we have the following result

\begin{theorem} \label{T1} Assume $(g_1)-(g_3)$. Then problem (\ref{BL}) has a positive solution. 
	
\end{theorem}

\section{Application: The zero  mass case}

In this section, we will show that  Theorem \ref{AT1} can be used to show the existence of solution for the zero mass case associated with (\ref{BL}). In this case, the conditions on $g$ are the following:  
$$
g(0)=0,  \quad \limsup_{s \to 0^+}\frac{g(s)}{s^{{2^*}-1}}=0, \leqno{(g_4)} 
$$
$$
\mbox{there is} \quad \xi>0 \, \, \mbox{such that} \,\, G(\xi)>0, \leqno{(g_5)}
$$
Let $\xi_0=\inf\{\xi>0\;G(\xi)>0\}$. If $g(s)>0$ for all $s > \xi_0$, then 
$$
\limsup_{s \to +\infty}\frac{g(s)}{s^{{2^*}-1}}=0. \leqno{(g_6)}
$$ 

The approach used in the last section can be repeated with $H^{1}(\mathbb{R}^N)$ and $H_{rad}^{1}(\mathbb{R}^N)$ replaced by $D^{1,2}(\mathbb{R}^N)$ and $D^{1,2}_{rad}(\mathbb{R}^N)$ respectively. In this section $\|\,\,\,\|$ denotes the usual norm in 
$D^{1,2}(\mathbb{R}^N)$, that is,
$$
\|u\|=\left(\int_{\mathbb{R}^N}|\nabla u|^{2}\, dx \right)^{\frac{1}{2}}.
$$

Arguing as Beresticki-Lions \cite{berest}, let us consider  
$$
g(s)=g_1(s)-g_2(s), \quad s \geq 0,
$$
where
$$
g_1(s)=g(s)^{+}=\max\{s,0\} \quad \mbox{and} \quad g_2(s)=g(s)^{-}=(-g(s))^{+}. 
$$
By definition of $g_1$ and $g_2$, we have that $g_1(s),g_2(s) \geq 0$  for all $s \geq 0$. Now, we extend the functions $g_1$ and $g_2$ to whole $\mathbb{R}$ as odd functions. The conditions on $g$ together with the definition of $g_1$ ensure that
$$
g_1(0)=0 \quad \mbox{and} \quad \lim_{s \to 0}\frac{g_1(s)}{|s|^{2^*-1}}=\lim_{|s| \to +\infty}\frac{g_1(s)}{|s|^{2^*-1}}=0.
$$

The above information combined with Strauss' Lemma implies in the lemma below
\begin{lemma} Let $(u_n) \subset D^{1,2}_{rad}(\mathbb{R}^N)$ be a sequence with $u_n\rightharpoonup u_0$ in $D^{1,2}_{rad}(\mathbb{R}^N)$. Then, 
$$
\lim_{n \to +\infty}\int_{\mathbb{R}^N}G_1(u_n)\,dx=\int_{\mathbb{R}^N}G_1(u_0)\,dx
$$	
and
$$
\lim_{n \to +\infty}\int_{\mathbb{R}^N}G_2(u_n)\,dx \geq \int_{\mathbb{R}^N}G_2(u_0)\,dx,
$$
where $G_i(s)=\int_{0}^{s}g_i(t)\,dt$ for $i=1,2$.
\end{lemma} 

In the present section, the functionals $\Phi, \Psi,J_h:D^{1,2}(\mathbb{R}^N) \to \mathbb{R}$ are as in Section 3. The reader is invited to observe that Lemmas \ref{L1} and \ref{L2} remain valid for the zero mass case, if we replace $H^{1}(\mathbb{R}^N)$ by $D^{1,2}(\mathbb{R}^N)$. Related to the Lemma \ref{L3} a little adjust must be done to prove that $J_h$ is coercive.

\begin{lemma} \label{L4} Assume that $(g_4)-(g_6)$ hold. Then  $J_h$ given in (\ref{jh}) is coercive. 
\end{lemma}
\begin{proof}
From  conditions on $g$, there is $c_1>0$ such that
$$
|G(s)|\leq c_1|s|^{2^*}, \quad \forall s \in \mathbb{R}.
$$
Thereby, by Sobolev embedding, 
$$
\Psi(u) \leq C |\nabla u|^{2^*}_{L^{2}(\mathbb{R}^N)},
$$
and so, 
$$
J_h(u)\geq \frac{1}{k}\left(\frac{1}{2}|\nabla u|^{2}_{L^{2}(\mathbb{R}^N)}\right)^{k}-C |\nabla u|^{2^*}_{L^{2}(\mathbb{R}^N)}.
$$
This shows that $J_h$ is coercive.  
\end{proof}

The previous results ensure that theorem below is true 

\begin{theorem} \label{T2} Assume that $(g_4)-(g_6)$ hold. Then (\ref{BL}) has a nontrivial solution. 
\end{theorem}

\subsection{A global minimization trick: Multiple solutions for a class of zero mass problems} \mbox{}\\
\mbox{}\,\,\,\,\,\, In this subsection we will show that the global minimization trick can be used to establish the existence of many solutions for a class of zero mass problems.  More precisely we will consider the problem (\ref{BL}), by supposing that $g:\mathbb{R} \to [0,+\infty)$ is a continuous function that satisfies the following conditions: \\

\noindent  $(g_7)$ \,\, There are $q >2^*$ and $c_1,c_2,\delta >0$ such that  
$$
c_1s^{q}\leq g(s) \leq c_2s^{q}, \quad \forall s \in [0,\delta].
$$
\noindent  $(g_8)$ $\displaystyle \lim_{s \to +\infty}\frac{g(s)}{s^{2^{*}-1}}=0.$

\vspace{0.5 cm}

By Theorem \ref{T2}, problem (\ref{BL}) has at least a positive solution. Now, our main goal is to prove that (\ref{BL}) has infinite many nontrivial solutions. To this end, we will employ the genus theory. Here as in the previous section, we extend $g$ to whole $\mathbb{R}$ as an odd function.

Before showing a version of Theorem \ref{AT1} involving multiplicity of solution, we will introduce more some notations. In what follows, we will say that a $C^{1}$-function $h:[0,+\infty) \to [0,+\infty)$ is {\bf  strongly compatible} with the $C^1$-functionals $\Phi,\Psi:X \to \mathbb{R}$, denoted by $h \in (\mathcal{SC})_{\Phi,\Psi}$, if $h \in \mathcal{C}_{\Phi,\Psi}$, the functional $J_h$ satisfies the $(PS)$ condition and the property below is also valid:  \\

\noindent $(AF_3)'$ \,\, For any sequence  $(u_j) \subset X \setminus \{0\}$ with $u_s \not= u_l$ for $s \not= l$ satisfying  
$$
\Phi'(u_i)-\lambda_i \Psi'(u_i)=0, \quad \forall i \in \mathbb{N}, 
$$
with $\lambda_i=(h'(\Phi(u_i)))^{-1}$, it is possible to find $v_i=v(\lambda_i,u_i) \in X \setminus \{0\}$ with $v_s \not= v_l$ for $s \not= l$ such that  
$$
\Phi'(v_i)-\Psi'(v_i)=0, \quad  \forall i \in \mathbb{N}.
$$

\begin{theorem} \label{AT2} Let $X$ be a reflexive Banach space and consider three functionals $J,\Phi,\Psi:X \to \mathbb{R}$ such that 
	$$
	J(u)=\Phi(u)-\Psi(u), \quad \forall u \in X,
	$$	
	with $\Phi \in (\mathcal{AF})_X$, $(\Phi,\Psi) \in (\mathcal{AP})$ and $\Phi,\Psi$ being even $C^1$-functionals. Moreover, assume that there is $h \in (\mathcal{SC})_{\Phi,\Psi}$ and that for each $j \in \mathbb{N}$ there is a sphere $S_r^{j-1} \subset E_j$ of radius $r_j=r(j)>0$ in a $j$-dimension space $E_j$  such that 
	$$
	\sup_{u \in S_{r_j}^{j-1}}J_h(u)<0.
	$$
	Then, functional $J$ has infinite many nontrivial solutions. 
\end{theorem} 
\begin{proof} \, By using the same arguments explored in the proof of Theorem \ref{AT1}, we know that $J_h$ is bounded from below. Since $J_h$ is even and satisfies the $(PS)$ condition , by Heinz \cite[Proposition 2.2]{H},  there exists a sequence of critical points $(u_j) \subset X$ for $J_h$ with
$$
c_j=J_h(u_j)<0 \quad \mbox{and} \quad c_j \to 0 \quad \mbox{as} \quad j \to +\infty. 
$$
Therefore,
$$
h'(\Phi(u_j))\Phi'(u_j)- \Psi'(u_j)=0, \quad \forall j \in \mathbb{N},
$$
and without lost of generality, we can assume that $u_j \not= u_l$ for $j \not= l$. Setting $\lambda_j=\frac{1}{h'(\Phi(u_j)}$, we get
$$
\Phi'(u_j)- \lambda_j\Psi'(u_j)=0, \quad  \forall j \in \mathbb{N}.
$$
Now, owning $(AF_3)'$, there is sequence $(w_j)$ that satisfies 
$$
\Phi'(w_j)- \Psi'(w_j)=0, \quad \forall j \in \mathbb{N},
$$
with $w_s \not= w_l$ for $s \not= l$, and the theorem follows.  
\end{proof}

In what follows, we are working with the functional $J_h:D^{1,2}(\mathbb{R}^N) \to \mathbb{R}$ defined by
$$
J_h(u)=\frac{1}{k}\left(\frac{1}{2}\int_{\mathbb{R}^N}|\nabla u|^{2}\,dx \right)^{k}-\int_{\mathbb{R}^N}F(u)\,dx,
$$
restricts to $D_{rad}^{1,2}(\mathbb{R}^N)$, because in this space we can use the Strauss Lemma to conclude that $J_h$ verifies the $(PS)$ condition. This choose is justified by the fact that it is easy to check that if $(u_n) \subset D_{rad}^{1,2}(\mathbb{R}^N)$ and $u_n \rightharpoonup u$ in $D_{rad}^{1,2}(\mathbb{R}^N)$, then
$$
\int_{\mathbb{R}^N}f(u_n)u_n\,dx \to \int_{\mathbb{R}^N}f(u)u\,dx,
$$
which is the key point to show the $(PS)$ condition. Moreover, it is very important point out that by Palais' principle of symmetric criticality \cite{Palais}, all of the critical points of $J_h$ in $D_{rad}^{1,2}(\mathbb{R}^N)$ are critical points of $J_h$ in $D^{1,2}(\mathbb{R}^N)$.

Since $J_h$ is even and bounded from below, in order to prove the existence of infinitely many solutions  for $J_h$, we need of the two lemmas below

\begin{lemma} \label{R1}
For each $j \in \mathbb{Z}$ there are a $j$-dimension space $E_j$ and a sphere $S_r^{j-1} \subset E_j$ of radius $r_j=r(j)>0$ such that 
$$
\sup_{u \in S_{r_j}^{j-1}}J_h(u)<0.
$$
\end{lemma}
\begin{proof}
To begin with, let us fix an orthogonal set of continuous functions $Y=\{\phi_1,...,\phi_j\} \subset D_{rad}^{1,2}(\mathbb{R}^N) $ and $E_j=span(Y) \subset D_{rad}^{1,2}(\mathbb{R}^N)$. Note that if $u \in S_r^{j-1} \subset E_j$, then 
$$
u=a_1\phi_1+...+a_j\phi_j
$$
and 
$$
\|u\|^2=\sum_{i=1}^{j}a_i^{2}=r^2.
$$  
For $r$ small enough, we know that $|u|_{\infty}$ is small, because $\|\,\,\,\|$ and $ |\,\,\, |_\infty$ are equivalent on $E_j$. Therefore, decreasing $r$ if necessary, we can assume that $|u(x)| \leq \delta$ for all $x \in \mathbb{R}^N$ and $u \in S_{r_j}^{j-1}$. Thereby, by $(g_7)$,  
$$
J_h(u)\leq \frac{r^{2k}}{2k}-c_1\int_{\mathbb{R}^N}|u|^{q}\,dx.
$$
On $E_j$, the norms $\|\,\,\,\|$ and $|\,\,\,\,|_{q}$ are also equivalent, then there is $c_3>0$ such that
$$
J_h(u)\leq \frac{r^{2k}}{2k}-c_3\|u\|^{q}=\frac{r^{2k}}{2k}-c_3r^{q}.
$$ 
Now, setting $k>q/2$, there exists $r>0$ small enough such that  
$$
J_h(u)<0, \quad \forall u \in S_r^{j-1},
$$ 
and this is precisely the assertion of the lemma. 
\end{proof}

\begin{lemma} Functional $J$ satisfies $(AF_3)'$.
	
\end{lemma}
\begin{proof}
Let $\lambda_j=(h'(\Phi(u_j)))^{-1}>0$ and $(u_j) \subset D^{1,2}(\mathbb{R}^N) \setminus  \{0\}$ satisfying
$$
\int_{\mathbb{R}^N}\nabla u_j \nabla w\,dx-\lambda_j \int_{\mathbb{R}^N}g(u_j)w\,dx=0, \quad \forall w \in H^{1}(\mathbb{R}^N),
$$
with $u_j \not= u_i$ for $j \not= i$. Then $v_j(x)=u_j(x/\sqrt{\lambda_j})$ is a nontrivial solution of    
$$
\int_{\mathbb{R}^N}\nabla v_j \nabla w\,dx-\int_{\mathbb{R}^N}g(v_j)w\,dx=0, \quad \forall w \in H^{1}(\mathbb{R}^N) \quad \mbox{and} \quad \forall j \in \mathbb{N}.
$$
\begin{claim} 
$v_j \not= v_i$ for $j \not= i$.	
\end{claim}	
The claim is immediate if $\lambda_j=\lambda_i$. Thus, it is enough to show that 
\begin{equation} \label{INE1}
v_j \not= v_i \quad \mbox{for} \quad j \not= i \quad \mbox{and} \quad \lambda_j \not= \lambda_i.
\end{equation}
Note that the equality $v_j=v_i$ yields  in the following one
$$
\left( \int_{\mathbb{R}^N}|\nabla u_j|^{2}\,dx \right)^{\frac{(k-1)(2-N)}{2}}\int_{\mathbb{R}^N}|\nabla u_j|^{2}\,dx=\left( \int_{\mathbb{R}^N}|\nabla u_i|^{2}\,dx \right)^{\frac{(k-1)(2-N)}{2}}\int_{\mathbb{R}^N}|\nabla u_i|^{2}\,dx,
$$
that leads to $\lambda_j=\lambda_i$, which is absurd. This proves the lemma.
\end{proof}

As an immediate consequence of the above analysis we have the following result

\begin{theorem} Assume $(g_7)-(g_8)$. Then (\ref{BL}) has infinitely many nontrivial solutions.	
\end{theorem}

\section{Application: The discontinuous case}

In this section we establish the existence of nontrivial solution for the problem
\begin{equation} \label{DP}
-\Delta u(x) \in \partial G(u(x)), \quad \mbox{a.e. in} \quad \mathbb{R}^N,
\end{equation}
where $N \geq 1$, $G$ is the primitive of a function $g:\mathbb{R} \to \mathbb{R}$ given by $g(s)=f(s)-s$, that is, 
$$
G(s)=\int_{0}^{s}g(t)\,dt=\int_{0}^{s}f(t)\,dt-\frac{1}{2}|s|^{2}=F(s)-\frac{1}{2}|s|^{2},
$$
and $\partial G(s)$ is the generalized gradient of $G$ at $s \in \mathbb{R}$, given by 
$$
\partial G(s)=[\underline{g}(s),\overline{g}(s)]
$$
where
$$
\underline{g}(s)=\displaystyle \lim_{r\downarrow 0}\mbox{ess
	inf}\left\{ g(t);|s-t|<r\right\} \quad \mbox{and} \quad \overline{g}(s)=\lim_{r\downarrow 0}\mbox{ess sup}\left\{
g(t);|s-t|<r\right\} .
$$

When $g$ is a continuous function, which is equivalent to say that $f$ is continuous, we know that $G \in C^{1}(\mathbb{R}, \mathbb{R})$, and in this case  
$$
\partial G(s)=\{g(s)\}, \quad \forall s \in \mathbb{R}.
$$

In this section, we assume that $f$ can have a finite number of  discontinuity points $a_1,a_2, ...., a_p \in (0,+\infty)$. Moreover, we are assuming the following conditions on $f$: \\ 

\noindent  $(f_1)$ \,\, $\displaystyle \lim_{s \rightarrow 0} \frac{f(s)}{s}=0$. \\

\noindent $(f_2)$ \,\, There are $A,B>0$ such that 
$$
|f(s)| \leq A|s|+B|s|^{q}, \quad \forall s \in \mathbb{R},
$$
for some $q \in (1, 2^{\ast}-1)$ where $2^{\ast}= \frac{2N}{N-2}$ if $N \geq 3$ and $2^*=+\infty$ if $N=1,2$.

\noindent $(f_3)$ \,\, $f(s)>0,$ \quad $\forall s>0$. \\

\noindent $(f_4)$ \,\, There is $\tau>0$ and $\tau \not= a_i $ for $i=1,2,..,p$ such that $F(\tau)-\frac{1}{2}|\tau|^{2}>0$. \\

Since we intend to find a nonnegative solution, in what follows we assume that 
$$
f(s)=0, \quad \forall s<0. 
$$

Hereafter, by a solution we understand as being a function $u \in W_{loc}^{2, \frac{q+1}{q}}(\mathbb{R}^{N}) \cap H^{1}(\mathbb{R}^{N})$ that verifies (\ref{DP}), or equivalently, the problem below
\begin{equation} \label{equivalente}
-\Delta u(x)  +u(x) \in \left[\underline{f}(u(x)), \overline{f}(u(x))\right], \quad \mbox{a.e. in} \quad \mathbb{R}^N.
\end{equation}
For the case where $f$ is a continuous function, the above solution must verify the equation 
$$
-\Delta u(x)  +u(x) = f(u(x)), \quad \mbox{a.e. in} \quad  \mathbb{R}^N.
$$

Following the same ideas explored in the previous section, we will consider the functional $J_h:H^{1}(\mathbb{R}^N) \to \mathbb{R}$ given by
$$
J_k(u)=\frac{1}{k}\left(\frac{1}{2}\int_{\mathbb{R}^N}|\nabla u|^{2}\,dx \right)^{k}+\frac{1}{2}\int_{\mathbb{R}^N}|u|^{2}\,dx-\int_{\mathbb{R}^N}F(u)\,dx,
$$
which is a locally Lipschitz functional. The reader is invited to see that Lemmas \ref{L2} and \ref{L3} of Section 3 still hold in the present section, however related to the Lemma \ref{L1} we need to do some modifications in its proof. From now on, $\Phi$ is as in (\ref{phi1}) and $\Psi$ is given by 
\begin{equation} \label{psi2}
\Psi(u)=\int_{\mathbb{R}^N}F(u)\,dx -\frac{1}{2}\int_{\mathbb{R}^N}|u|^{2}\,dx, \quad \forall u \in H^{1}(\mathbb{R}^N).
\end{equation}

\begin{lemma} \label{L5} Let $\lambda >0$ and $u \in H^{1}(\mathbb{R}^N)$ such that $0 \in \Phi'(u) - \lambda\partial \Psi(u)$, where $\Phi$ and $\Psi$ were given in (\ref{phi1}) and (\ref{psi2}) respectively.  Setting $v(x)=u(x/\sqrt{\lambda})$, we have that $0 \in \Phi'(v) - \partial \Psi(v)$. Hence, $(\Phi,\Psi) \in (\mathcal{AP})$.
\end{lemma}
\begin{proof} By hypothesis $0 \in  \Phi'(u) - \lambda\partial \Psi(u)$.  By \cite{Chang, clarke}, we know that the generalized gradient of $\Psi$ at $u$  is the set 
	$$
	\partial \Psi(u)=\left\{ \mu \in  (H^{1}(\mathbb{R}^N))^*\,:\, \Psi^{0}(u,v) \geq \langle \mu,w \rangle, \, \forall w \in H^{1}(\mathbb{R}^N) \right\},
	$$
	where $\Psi^{0}(u,v)$ denotes the directional derivative of $\Psi$ on $u$ in the direction of $v \in H^{1}(\mathbb{R}^N)$, which is defined by
	$$
	\Psi^{0}(u,v)=\limsup_{\xi \to 0, \tau \downarrow 0}\frac{\Psi(u+\xi+\tau w)-\Psi(u+\xi)}{\tau}.
	$$
From this, 
\begin{equation} \label{Z0}
\limsup_{\xi \to 0, \tau \downarrow 0}\frac{\Psi(u+\xi+\tau w)-\Psi(u+\xi)}{\tau} \geq \frac{1}{\lambda}\int_{\mathbb{R}^N}\nabla u \nabla w\,dx, \quad \forall w \in H^{1}(\mathbb{R}^N).
\end{equation} 
Setting $v(x)=u(x/\sqrt{\lambda}), w_\lambda(x)=w(x/\sqrt{\lambda})$ and $\xi_\lambda(x)=\xi(x/\sqrt{\lambda})$, a change variable permits to rewrite (\ref{Z0}) of the form
\begin{equation} \label{Z1}
\limsup_{\xi_\lambda \to 0, \tau \downarrow 0}\frac{\Psi(v+\xi_\lambda+\tau w_\lambda)-\Psi(v+\xi_\lambda)}{\tau} \geq \int_{\mathbb{R}^N}\nabla v \nabla w_\lambda\,dx, \quad \forall w_\lambda \in H^{1}(\mathbb{R}^N),
\end{equation} 
from where it follows that
$$
0 \in \Phi'(v) - \partial \Psi(v).
$$
This shows the desired result. 
\end{proof}	

From the above study  we derive the following result
\begin{theorem} \label{T3} Assume $(f_1)-(f_4)$. Then (\ref{BL}) has a nontrivial solution. 
\end{theorem}

\section{Application: The elliptic system case}

In this section, we will show a version of the  global minimization trick that can be used to prove the existence of nontrivial solution for the following class of elliptic systems 
\begin{equation} \label{ES}
\left\{
\begin{array}{l}
-\Delta u=F_u(u,v), \quad \mbox{in} \quad \mathbb{R}^N,  \\
-\Delta v=F_v(u,v), \quad \mbox{in} \quad  \mathbb{R}^N, \\
\end{array}
\right.
\end{equation}
where $F:\mathbb{R}^2 \to \mathbb{R}$ is a $C^{1}$-function such that
$$
\limsup_{|(s,t)|\to 0}\frac{F(t,s)}{|(t,s)|^{2}}\leq -m<0. \leqno{(F_1)}
$$ 
There is $q \in(2,2^*)$ such that 
$$
\limsup_{|(s,t)|\to 0}\frac{|F(t,s)|}{|(t,s)|^{q}}<+\infty. \leqno{(F_2)}
$$ 
There is $(t_0,s_0) \in \mathbb{R}^2$ such that 
$$
F(t_0,s_0)>0. \leqno{(F_3)}
$$
The energy functional $J:X=H^{1}(\mathbb{R}^N) \times H^{1}(\mathbb{R}^N) \to \mathbb{R} $ associated with (\ref{ES}) is given by 
$$
J(u,v)=\frac{1}{2}\int_{\mathbb{R}^N}(|\nabla u|^{2}+|\nabla v|^{2})\,dx-\int_{\mathbb{R}^N}F(u,v)\,dx.
$$

In order to use the Theorem \ref{AT1} in the present case, it is necessary to do some adjusts. Here, the idea is the following: From now on, we say that a $C^1$-functional $J:X \to \mathbb{R}$ is ${\bf C^1}$-{\it {\bf Almost Lower Semicontinuous}}  with relation to $X$, denoted by $J \in C^{1}(\mathcal{ALS})_X$, if $J$ is bounded from below in $X$ and for any minimizing sequence $(u_n) \subset X$ for $J$, there exist a $(PS)$ sequence $(v_n) \subset X$ and $v \in X \setminus \{0\}$ with $v_n \rightharpoonup v$ in $X$ such that
$$
J(u_n) \geq J(v_n) \geq \inf_{u \in X}J(u), \quad \forall n \in \mathbb{N}
$$ 
and
$$
J'(v_n) \to 0.
$$
Moreover, we will say that a $C^{1}$-function $h:[0,+\infty) \to [0,+\infty)$ is {\bf *-compatible} with the functionals $\Phi,\Psi:X \to \mathbb{R}$, denoted by $h \in \mathcal{C}^*_{\Phi,\Psi}$, when $h \in \mathcal{C}_{\Phi,\Psi}$ replacing $(\mathcal{ALS})_X$ by $C^{1}(\mathcal{ALS})_X$ and $J'_h$ is weakly sequentially continuous $(\mathcal{WSC})$, that is, 
$$
v_n\rightharpoonup v \quad \mbox{in} \quad X \Rightarrow  J_h'(v_n) \stackrel{*}{\rightharpoonup} J_h'(v) \quad \mbox{in} \quad X^*,
$$ 
more precisely, 
$$
v_n\rightharpoonup v \quad \mbox{in} \quad X \Rightarrow  J_h'(v_n)\phi \stackrel{*}{\rightharpoonup} J_h'(v)\phi, \quad \forall \phi \in X.
$$ 
As an immediate consequence of $(\mathcal{WSC})$, we have $J_h'(v)=0$, that is, $v$ is a critical point of $J_h$.

\begin{theorem} \label{AT3} Let $X$ be a reflexive Banach space and consider three $C^1$-functionals $J,\Phi,\Psi:X \to \mathbb{R}$ such that 
	$$
	J(u)=\Phi(u)-\Psi(u), \quad \forall u \in X,
	$$	
	with $(\Phi,\Psi) \in (\mathcal{AF})_X$. If $\mathcal{C}^*_{\Phi,\Psi}$ is not empty, then functional $J$ has a nontrivial critical point.
\end{theorem} 
\begin{proof}\, In what follows we fix $h \in \mathcal{C}^*_{\Phi,\Psi}$ and the functional $J_h:X \to \mathbb{R}$ defined by
$$
J_h(u)=h(\Phi(u))-\Phi(u), \quad \forall u \in X.
$$
Arguing as in the proof of Theorem \ref{AT1}, there is a sequence $(u_n) \subset X$ satisfying
$$
\lim_{n \to +\infty}J_h(u_n)=\inf_{u \in X}J_h(u)=J^{\infty}_h<0.
$$
From $C^{1}(\mathcal{ALS})_X$, there is $(v_n) \subset X$ such that 
$$
\lim_{n \to +\infty}J_h(v_n)=J^{\infty}_h \quad \mbox{and} \quad  \lim_{n \to +\infty}J'_h(v_n)=0. 
$$
Since $X$ is reflexive, we can assume that $v_n \rightharpoonup v$ in $X$. This together with the fact that $J_h' \in (\mathcal{WSC})$ gives  $J_h'(v_n) \stackrel{*}{\rightharpoonup} J_h'(v)$. Then, $J_h'(v)=0$ and $v$ is a critical point of $J_h$. By using $(AF_3)$, let us deduce that $J$ has a nontrivial critical point. 
\end{proof}

After the above study we are ready to apply Theorem \ref{AT3} by considering the functionals 
\begin{equation} \label{phi3}
\Phi(u,v)=\frac{1}{2}\int_{\mathbb{R}^N}(|\nabla u|^{2}+|\nabla v|^{2})\,dx,
\end{equation}
\begin{equation} \label{psi3}
\Psi(u,v)=\int_{\mathbb{R}^N}F(u,v)\,dx
\end{equation}
and 
\begin{equation} \label{jr2}
J_h(u,v)=\frac{1}{k}\left(\frac{1}{2}\int_{\mathbb{R}^N}(|\nabla u|^{2}+|\nabla v|^{2})\,dx\right)^{k}-\int_{\mathbb{R}^N}F(u,v)\,dx,
\end{equation}
for $k > \frac{N}{N-2}$. A simple computation shows that $\Phi$ and $\Psi$ verify $(AF_1)-(AF_2)$.  The lemma below ensures that $(AF_3)$ also holds.

\begin{lemma} \label{L5} The functionals $\Phi$ and $\Psi$ given in (\ref{phi3}) and (\ref{psi3}) belongs to $(\mathcal{AP})$. 
	
\end{lemma}
\begin{proof} Let $\lambda >0$ be a real number satisfying 
$$
\Phi'(u,v)-\lambda \Psi'(u,v)=0,
$$	
for some $(u,v) \in X$. Then, $(u,v)$ is a solution of the following system 
\begin{equation} \label{ES1}
\left\{
\begin{array}{l}
-\Delta w= \lambda F_w(w,z), \quad \mbox{in} \quad \mathbb{R}^N, \\
-\Delta z=\lambda F_z(w,z), \quad \mbox{in} \quad  \mathbb{R}^N. \\
\end{array}
\right.
\end{equation}
Setting $(\phi(x),\psi(x))=(w(x/\sqrt{\lambda}),z(x/\sqrt{\lambda}))$, it is easy to see that $(\phi,\psi)$ is a solution of the system
\begin{equation} \label{ES2}
\left\{
\begin{array}{l}
-\Delta \psi= F_\phi(\phi,\psi), \quad \mbox{in} \quad \mathbb{R}^N, \\
-\Delta \psi= F_\psi(\phi,\psi), \quad \mbox{in} \quad  \mathbb{R}^N, \\
\end{array}
\right.
\end{equation}
and so, 
$$
\Phi'(\phi,\psi)-\Psi'(\phi,\psi)=0,
$$	
finishing the proof. 
\end{proof}	

\begin{lemma} \label{L5} The functional $J_h$ belongs to $C^{1}(\mathcal{ALS})_X$.
	
\end{lemma}	
\begin{proof} 
First of all, we know that 
$$
\inf_{(u,v) \in X}J_h(u,v)=J^{\infty}_h<0.
$$
Hence, there is a minimizing sequence $((u_n,v_n)) \subset X$ such that 
$$
J_h(u_n,v_n) \to J^{\infty}_h.
$$
By using Ekeland Variational principal, we can assume that $((u_n,v_n))$ is a $(PS)_{J^{\infty}_h}$ sequence for $J_h$, that is,  
$$
J_h(u_n,v_n) \to J^{\infty}_h \quad \mbox{in} \quad \mathbb{R}
$$
and 
$$
J'_h(u_n,v_n) \to 0 \quad \mbox{in} \quad X^*.
$$
By Lions \cite{Lions2}, there are $(y_n) \subset \mathbb{R}^N$ and $R,\beta>0$ such that
\begin{equation} \label{W1}
\int_{B_R(y_n)}(|u_n|^{2}+|v_n|^{2})\,dx\geq \beta, \quad \forall n \in \mathbb{N}.
\end{equation}
By using this information and setting  $(w_n(x),z_n(x))=(w_n(x+y_n),z_n(x+y_n))$, we derive that 
$$
J_h(w_n,z_n) \to J^{\infty}_h \quad \mbox{in} \quad \mathbb{R}
$$
and 
$$
J'_h(w_n,z_n) \to 0 \quad \mbox{in} \quad X.
$$
Since $((u_n,v_n))$ is a bounded sequence, we also have that $((w_n,z_n))$ is also bounded in $X$, and so, we can assume that $(w_n,z_n) \rightharpoonup (w,z)$ for some $(w,z) \in X$. Moreover, by (\ref{W1}), we derive that
$$
\int_{B_R(0)}(|w|^{2}+|z|^{2})\,dx\geq \beta,
$$ 
and so $(w,z) \not= (0,0)$. This proves that $J_h \in C^{1}(\mathcal{ALS})_X$.
\end{proof}

The above study guarantees that the result below holds
\begin{theorem} \label{T4} Assume $(F_1)-(F_3)$. Then (\ref{ES}) has a nontrivial solution. 
\end{theorem}

\section{Final comments}

The Global minimization trick can be used for other classes of problems, here we have cited only some of them, for example, we can apply  the abstract theorems showed in the present article to prove the existence of nontrivial solution for the following problem:  
$$
(- \Delta)^{\alpha}u = g(u), \quad \mbox{in} \quad \mathbb{R}^N,
$$ 
with $\alpha \in (0,1)$ and $N \geq 1$.

\end{document}